\documentclass[a4paper,oneside]{amsart}
\usepackage[utf8]{inputenc}
\usepackage[vscale=0.8]{geometry}
\usepackage{mathtools,amssymb,xcolor,upref,url}

\definecolor{darkred}{rgb}{0.6,0,0}
\definecolor{darkgreen}{rgb}{0,0.5,0}
\definecolor{darkmagenta}{rgb}{0.5,0,0.5}
\usepackage[pdftex, colorlinks=true,
  linkcolor=darkred,
  citecolor=darkgreen,
  urlcolor=darkmagenta]{hyperref}

\makeatletter
\let\@ltx@cite=\@cite
\renewcommand\@cite[2]{\textup{\@ltx@cite{#1}{#2}}}
\makeatother

\DeclarePairedDelimiter\abs{\lvert}{\rvert}
\DeclarePairedDelimiter\norm{\lVert}{\rVert}
\DeclarePairedDelimiter\pars()
\DeclarePairedDelimiter\bracks\lbrack\rbrack
\newcommand\iverson[2][]{\bracks[#1]{\mkern1mu #2 \mkern1mu}}

\providecommand\given{}
\newcommand\SetSymbol[1][]{\nonscript\:#1\vert\allowbreak
   \nonscript\:\mathopen{}}
\DeclarePairedDelimiterX\Set[1]\{\}{%
   \renewcommand\given{\SetSymbol[\delimsize]}
   #1}

\newcommand\lAngle[1][]{\mathopen{\ooalign{$#1\langle$\hidewidth\cr$\mkern2mu#1\langle$}}}
\newcommand\rAngle[1][]{\mathclose{\ooalign{$#1\rangle$\hidewidth\cr$\mkern2mu#1\rangle$}}}
\newcommand\gauge[2][]{\lAngle[#1]#2\rAngle[#1]}
\newcommand\vVert[1][]{{\ooalign
    {$#1\vert$\hidewidth\cr$\mkern2.5mu#1\vert$\hidewidth\cr$\mkern5mu#1\vert$}}}
\newcommand\altnorm[2][]{\mathopen{\vVert[#1]}#2\mathclose{\vVert[#1]}}

\newcommand{\C}{\mathbb C}
\newcommand{\R}{\mathbb R}
\newcommand{\dott}{\, \cdot\,}
\DeclareMathOperator{\co}{co}
\DeclareMathOperator{\cco}{\overline{co}}

\newtheorem{theorem}{Theorem}[section]
\newtheorem*{theorem*}{Theorem}
\newtheorem{proposition}[theorem]{Proposition}
\newtheorem{lemma}[theorem]{Lemma}
\newtheorem*{lemma*}{Lemma}
\theoremstyle{definition}
\newtheorem{definition}[theorem]{Definition}
\newtheorem{definitions}[theorem]{Definitions}
\newtheorem*{remark*}{Remark}

\numberwithin{equation}{section}

\allowdisplaybreaks

\title{The Aubin--Lions--Dubinski\u{\i} theorems on compactness in Bochner spaces}

\author[Hanche-Olsen]{Harald Hanche-Olsen}
\address[Hanche-Olsen]{\newline
    Department of Mathematical Sciences,
   NTNU  Norwegian University of Science and Technology,
    NO--7491 Trondheim, Norway}
\email[]{\href{harald.hanche-olsen@ntnu.no}{harald.hanche-olsen@ntnu.no}} 
\urladdr{\href{https://folk.ntnu.no/hanche/}{https://folk.ntnu.no/hanche/}}

\author[Holden]{Helge Holden}
\address[Holden]{\newline
    Department of Mathematical Sciences,
   NTNU  Norwegian University of Science and Technology,
    NO--7491 Trondheim, Norway}
\email[]{\href{helge.holden@ntnu.no}{helge.holden@ntnu.no}} 
\urladdr{\href{https://www.ntnu.edu/employees/helge.holden}{https://www.ntnu.edu/employees/helge.holden}}

\date{\today} 



\thanks{Published in {\em Pure and Applied Functional Analysis}, Vol.~9, No 5, pp. 1133--1144 (2024). The research of HH was supported in part by the 
project \textit{IMod --- Partial differential equations, statistics and data:
An interdisciplinary approach to data-based modelling}, project number 65114, from the Research Council of Norway, and  the Swedish Research Council under grant no. 2021-06594 while HH was in residence at Institut Mittag-Leffler in Djursholm, Sweden during the fall semester of 2023.}

\dedicatory{Dedicated to Fritz Gesztesy with admiration on the occasion of his 70th birthday}

\begin{document}

\begin{abstract}
A fundamental issue in the theory of time-dependent differential equations
is to characterize precompact sets in Bochner spaces.

We here survey the theory, starting with the classical Aubin--Lions inequality
and its important extension by Dubinski\u{\i}.
In particular, we give a simple and self-contained proof
of the compactness result due to Chen, J\"ungel, and Liu.
\end{abstract}

\maketitle
\section{Introduction} \label{sec:intro}

A common question in the theory of time-dependent differential equations
regards sufficient conditions for precompactness
of sets in a Bochner space $L^p(0,T;B)$,
where $B$ is a  Banach space and $(0,T)$ is a time interval.
Recall that a subset $S$ of a metric space is called \emph{precompact}
if every sequence in $S$ has a convergent subsequence
(with a limit not necessarily in $S$).
Equivalently, the closure of $S$ is compact.

The classical Aubin theorem \cite{aubin-1963} from 1963 states that
if $f_n$ is bounded in $L^p(0,T;X)$ where $X\subset B$ is compactly embedded,
and the time derivative $\partial f_n/\partial t$ is bounded in  $L^p(0,T;Y)$,
where $B\subset Y$ is continuously embedded,
then the sequence $(f_n)$ is relatively compact in $L^p(0,T;B)$.
A general discussion of this topic can be found
in the classical and widely cited paper by Simon \cite{simon-1987}.
Later on, several generalizations and extensions have been proved.
Covering them all is beyond the scope of this paper.

A core part of Aubin's argument is based on
what is frequently called the Aubin--Lions lemma.
We prefer to call it the Aubin--Lions \emph{inequality}
to avoid confusion with the compactness theorem
for which it is an important ingredient.
Note that some authors (e.g., \cite{chen-jungel-liu-2014})
  refer to the resulting compactness theorem as the Aubin–Lions lemma.

\begin{lemma*}[Aubin–Lions inequality]
  Let $X\subseteq B\subseteq Y$ with a compact embedding of $X$ in $B$
    and continuous embedding of $B$ in $Y$.
    Then for all $\varepsilon>0$ there exists an $\eta$ such that for all $v\in X$
  \begin{equation*}
    \norm{v}_B\le \varepsilon\norm{v}_X+ \eta \norm{v}_Y.
  \end{equation*}
\end{lemma*}
The first proof can be found in Lions \cite[Prop.~4.1, p.~59]{lions-1961},
see also  \cite[Lemma 5.1, p.~58]{lions-1969}.

\begin{remark*}
  Some authors, see, e.g., \cite{MR0209639,barrett-suli-2012},
  refer to this as Ehrling's lemma.
  Also Nirenberg \cite{nirenberg} quotes Ehrling's result.
  In Ehrling \cite[eq.~(6)]{ehrling-1954}
  we find the following  concrete estimate, resembling the above inequality:
  \begin{equation}\label{eq:ehr1}
    \norm{f}_{L^2(\partial D)}^2\le A\big(h \norm{f}_{\dot H^1(D)}^2+h^{-1}\norm{f}_{L^2(D)}^2 \big)
  \end{equation}
  corresponding to $X=\dot H^1(D)$, $B=L^2(\partial D)$,  and $Y=L^2(D)$ with $h$ positive.
  Here, $\dot H$ denotes a homogeneous Sobolev space.
  However, these spaces do not have the above embedding properties.
  Using this on the derivative yields \cite[eq.~(7)]{ehrling-1954}
  \begin{equation*}
    \norm{f}_{\dot H^1(\partial D)}^2
    \le A\big(h \norm{f}_{\dot H^2(D)}^2+h^{-1}\norm{f}_{\dot H^1(D)}^2 \big).
  \end{equation*}
  Furthermore, he gets \cite[p.~272, line 7]{ehrling-1954}
  \begin{equation*}
    \norm{f}_{\dot H^1(D)}^2\le A\big(h \norm{f}_{\dot H^{2}(D)}^2 +h^{-1}\norm{f}_{L^2(D)}^2\big).
  \end{equation*}
  Estimate \eqref{eq:ehr1} appears as a ``Peter--Paul'' version
  of the standard Sobolev estimate for traces,
  namely $\norm{f}_{L^2(\partial D)}\le C\norm{f}_{H^1(D)}$.
  Since these estimates do not have the relevant embedding properties
  and are stated for very concrete spaces,
  the term “Ehrling's lemma” does not seem justified.
\end{remark*}

Shortly after the appearance of Aubin's paper, Dubinski\u{\i}
\cite[Lemma 1, p.~229 of \cite{MR0221883}]{dubinskii-1965}%
\footnote{\cite{MR0221883} contains the English translation of \cite{dubinskii-1965}},
see also \cite[Lemma 12.1, p.~141]{lions-1969},
generalized the Aubin--Lions inequality in a very useful direction.
See \cite{barrett-suli-2012} for a nice review
and correction of Dubinski\u{\i}'s argument.
His approach replaced the Banach space $X$ with a cone $M$,
being a subset of a vector space that is closed
with respect to multiplication with nonnegative scalars,
and the norm by a non-negative homogeneous scalar function
(a \emph{gauge} in our preferred terminology)
that only vanishes for a zero argument.
His aim was to show existence of unique weak solutions
of degenerate parabolic equations, an example being
\begin{equation*}
  u_t=\nabla\cdot\big(\abs{u}^\gamma \nabla u\big)+ h,
\end{equation*}
on a bounded domain in $\R^d$ with Dirichlet boundary conditions.
In order to accomplish this,
he proved the generalization of the Aubin--Lions inequality below.

We first define a gauge $[u]_M$ on a cone $M$
as a map $[\dott]_M\colon M\to [0,\infty)$
such that $[u]_M=0$ if and only if $u=0$.
In addition we require $[\lambda u]_M=\lambda [u]_M$ for all $\lambda\ge 0$.

\begin{lemma*}[Aubin--Lions–Dubinski\u{\i} inequality]
  Let $A_0, A_1$ be normed spaces, and let $M$ be a cone with gauge $[\dott]_M$.
  Assume that $M\subset A_0 \subset A_1$ with  $M\subset A_0$ compact
  and  $A_0 \subset A_1$ continuous embeddings, respectively.
  Then for any $\varepsilon>0$ there exists an $N(\varepsilon)$
  such that  for all $u,v\in M$ we have
  \begin{equation*}
    \norm{u-v}_{A_0}\le \varepsilon([u]_M+ [v]_M)+ N(\varepsilon) \norm{u-v}_{A_1}.
  \end{equation*}
\end{lemma*}

\begin{remark*}
  Dubinski\u{\i} called $M$
  a seminormed non-negative cone
  with $[\dott]_M$ a seminorm,
  but we have used our preferred terminology instead.
  Precise definitions will be given later.
  He gives an example of a gauge   as
  \begin{equation*}
    [u]_M^{\alpha+\beta}= \int_D \abs{u}^\alpha \abs{\nabla u}^\beta dx+ \int_{\partial D}\abs{u}^{\alpha+\beta}ds,
  \end{equation*}
  for $\alpha>0$, $\beta\ge 1$ and $D$ a bounded domain in $\R^d$ with smooth boundary $\partial D$.
  The cone $M$ was the set of $u$ satisfying $[u]_M<\infty$.
  It is not closed under addition
  (equivalently for cones, not convex).
\end{remark*}

Dubinski\u{\i} showed the following compactness result,
shown here as corrected and generalized by Barrett and Süli:

\begin{theorem*}[{\cite[Thm.~1, p.\ 229]{dubinskii-1965},
    \cite[Thm.\ 2.1]{barrett-suli-2012}}]
  Let $A_0, A_1, M$ be as in the lemma,
  and $p$, $p_1 \in [0,\infty]$ with $(p,p_1) \ne (\infty,1)$.
  Further, let $\mathcal{Y}$ be the set of Bochner measurable functions $(0,T) \to M$
  having a weak derivative $u_t$ that is also Bochner measurable, such that
  \begin{equation*}
    [u]_{\mathcal{Y}}
    := \pars[\bigg]{\int_0^T [u]_M^p\,dt}^{1/p}
    + \pars[\bigg]{\int_0^T \norm{u_t}_{A_1}^{p_1}}^{1/p_1}
  \end{equation*}
  (with the obvious modifications if $p=\infty$ or $p_1=\infty$)
  is finite.
  Then $\mathcal{Y}\subset L^p((0,T); A_0)$ is a compact embedding.
\end{theorem*}

Dubinski\u{\i}'s inequality was further generalized in \cite{chen-jungel-liu-2014},
where the authors realized that one does not need the continuous embedding $A_0\subset A_1$.
More precisely they showed the following result.

\begin{lemma*}
  Let $B,Y$ be Banach spaces, and let $M$ be a cone in $B$ with gauge $[\dott]_M$.
  Assume that $M\subset B$ is compactly embedded.
  Assume that for all $w_n\in B\cap Y$
  such that $w_n\to w$ in $B$ and $w_n\to 0$ in $Y$  we have that $w=0$.
  Then we have that for any $\varepsilon>0$ there exists an $N(\varepsilon)$
  such that  for all $u,v\in M$ we have
  \begin{equation*}
    \norm{u-v}_{B}\le \varepsilon([u]_M+ [v]_M)+ N(\varepsilon) \norm{u-v}_{Y}.
  \end{equation*}
\end{lemma*}

Under these conditions, they showed the following compactness result.

\begin{theorem*}
  Let $B, Y, M$ be as in the lemma, and assume that $M\cap Y\neq \emptyset$.
  Let $p\in [1,\infty]$.
  Assume the conditions of the lemma are satisfied.
  Let $U\subset L^p((0,T);M\cap Y)$ be bounded in $L^p((0,T);M)$.
  Furthermore, assume that translations are uniformly continuous, that is,
  $\norm{{\sigma_h u-u}}_{L^p((0,T-h);Y)}\to 0$ uniformly for $u \in U$ as $h\to 0$.

  Then $U$ is relatively compact in $L^p((0,T);B)$.
\end{theorem*}

Our interest in these questions arose from the obvious analogy
with precompact subsets of Lebesgue spaces.
The classical theorem of Kolmogorov--Riesz and its improvement by Sudakov
give a complete  characterization of precompact subsets of Lebesgue spaces,
see \cite{h-o-h-2010, h-o-h-2010a, h-o-h-m-2019}.
In the present paper, we draw on the analogy between these questions.
In particular, we shall prove versions of the results referenced above
for Bochner spaces based on $\R^n$ rather than an interval $(0,T)$.
We strive for simplicity of proof and digestible and self-contained exposition.
We also suggest and use terminology somewhat different
from what is seen in the literature so far.

\section{Compactness for Bochner spaces based on Euclidean spaces} \label{sec:comp}

We start with a very brief overview
of the basic theory of Bochner spaces.
Consider a $\sigma$-finite measure space $(\Omega,\mu)$,
and a Banach space $B$.
A function $f \colon \Omega \to B$ is called \emph{Bochner measurable}
if it is the $\mu$-almost everywhere limit
of a sequence of simple functions,
a simple function being one that takes only a finite number of values,
each on a measurable subset of $\Omega$ with finite measure.
The Pettis measurability theorem \cite[Thm.~1.1]{pettis-1938}
states that $f$ is Bochner measurable if, and only if,
it is weakly measurable and separably-valued.
Here, $f$ is called \emph{weakly measurable}
if the composition of $f$
with any bounded linear functional is measurable,
and \emph{separably-valued} if there is a separable subspace $B' \subseteq B$
so that $f(x) \in B'$ for $\mu$-almost every $x \in \Omega$.
Pettis considered only the case where $\mu(\Omega)<\infty$,
but the $\sigma$-finite case follows easily.

The \emph{Bochner space} $L^p(\Omega,\mu;B)$ (with $1 \le p < \infty$) consists of
all Bochner measurable functions $f$
satisfying $\int_{\Omega} \norm{f}^p \,d\mu < \infty$.
With the norm $\norm{f}_p = \pars[\big]{\int_{\Omega} \norm{f}^p \,d\mu}^{1/p}$, the space
$L^p(\Omega,\mu;B)$ becomes a Banach space.

A note on notation: Whenever $f \colon \Omega \to B$,
we let $\norm{f}$ denote the \emph{function} $t \mapsto \norm{f(t)}$.
Think of it as “pointwise norm”.
Accordingly, we never omit the subscript $p$
on the $L^p$-norm $\norm{f}_p$.

It turns out that the simple functions are dense in $L^1(\Omega,\mu;B)$.
This fact allows the definition of the \emph{Bochner integral}
$\int_{\Omega} f \,d\mu$ for all $f \in L^1(\Omega,\mu;B)$ by continuity,
starting from the obvious definition of the integral for simple $f$.
Elementary properties like the \emph{integral triangle inequality}
$\norm{\int_{\Omega} f \,d\mu} \le \int_{\Omega} \norm{f} \,d\mu$
are immediate consequences.

The totally bounded subsets of $L^p(\Omega,\mu;B)$
were characterized by Diaz and Mayoral in \cite{diaz-mayoral-1999};
see \cite{van-neerven-2014} for an elementary proof.
This characterization is not well suited for applications in PDE theory, however.

From now on we concentrate our attention
on the Bochner spaces $L^p(\R^d;B)$
(the use of Lebesgue measure is hidden in the notation).
We write $\abs{Q}$ for the Lebesgue measure of a set $Q \subseteq \R^d$.

Our concern is with precompactness,
but let us turn our attention to a closely related concept.
Recall that a subset $S$ of a metric space $X$ is called \emph{totally bounded}
if for every $\varepsilon>0$, $S$ can be covered by a finite number of sets,
each of diameter less than $\varepsilon$ (an $\varepsilon$-cover).
Equivalently, for each $\varepsilon>0$, there is a finite subset of $S$ (an $\varepsilon$-net)
so that every member of $S$ is closer than $\varepsilon$ to some member of the subset.
It is well known that a metric space is compact
if and only if it is complete and totally bounded.
It follows that a subset of a complete metric space
is precompact if and only if it is totally bounded.
In applications, one wants convergent subsequences,
i.e., one wants precompactness.
However, from now on we shall concentrate on total boundedness instead,
simply because that is what emerges from the proofs.

In the following definitions,
$\mathcal{F}$ is a subset of $L^p(\R^d;B)$, where $1 \le p < \infty$.

\begin{definition}[\cite{krukowski-2023}] \label{def:Lp-ev}
  $\mathcal{F}$ is called
  \emph{$L^p$-equivanishing}
  if $\int_{\R^d} \iverson{\abs{x}>r}\,\norm{f(x)}^p\,dx \to 0$
  when $r \to \infty$, uniformly for $f \in \mathcal{F}$.
  Here and later we employ the “Iverson bracket” \cite{knuth-1992}:
  When $S$ is a statement, $\iverson{S}=1$ if $S$ is true,
  and $\iverson{S}=0$ if $S$ is false.
\end{definition}

\begin{definition}[\cite{krukowski-2023}] \label{def:Lp-ec}
  $\mathcal{F}$ is called
  \emph{$L^p$-equicontinuous}
  if $\norm{\sigma_{h}f-f}_p \to 0$
  when $\abs{h} \to 0$, uniformly for $f \in \mathcal{F}$.
  Here $\sigma_hf(x)=f(x+h)$ ($x$, $h \in \R^d$).
\end{definition}

It is easily seen that a totally bounded set
is $L^p$-equivanishing and $L^p$-equicontinuous,
since it is true of any singleton set
and hence of any finite set.

\begin{definition} \label{def:tbm}
  $\mathcal{F}$ is called
  \emph{totally bounded in the mean}
  if for each bounded measurable set $E \subset \R^d$,
  $\Set{\int_E f(x)\,dx \given f \in \mathcal{F}}$ is totally bounded in $B$.
\end{definition}

A totally bounded set is totally bounded in the mean
because the map $f \mapsto \int_E f(x)\,dx$ is uniformly continuous.

The following theorem is due (at least) to Aubin and Simon
in the case of functions supported on a bounded interval.
(In that case, $L^p$-equivanishing is of course irrelevant.)
It may also be considered a variant of the Kolmogorov–Riesz theorem
– and indeed, our proof is a straightforward adaptation
of the proof of the Kolmogorov–Riesz theorem
given in \cite{h-o-h-2010}.
All the compactness results below
will be proved by reducing them to this theorem.

\begin{theorem} \label{thm:kras} 
  Let $B$ be a Banach space,
  and $\mathcal{F} \subset L^p(\R^d;B)$ a subset with $1 \le p < \infty$.
  Then $\mathcal{F}$ is totally bounded if and only if
  it is $L^p$-equivanishing, $L^p$-equicontinuous,
  and totally bounded in the mean.
\end{theorem}

We shall prove the theorem using the following lemma.
We omit the trivial proof.
The reader can probably construct one
in less time than it would take to look it up in \cite{h-o-h-2010}.

\begin{lemma}[{\cite[Lemma 1]{h-o-h-2010}}] \label{lemma:tb}
  Let $X$ be a metric space.
  Assume that, for every $\varepsilon>0$, there exists some
  $\delta>0$, a metric space $W$\!, and a mapping
  $\Phi\colon X\to W$ so that $\Phi[X]$ is totally bounded, and
  whenever $x,y\in X$ are such that
  $d\big(\Phi(x),\Phi(y)\big)<\delta$, then $d(x,y)<\varepsilon$.
  Then $X$ is totally bounded.
\end{lemma}

\begin{proof}[Proof of Theorem \ref{thm:kras}]
  The necessity of the three conditions
  was dealt with above.

  Now assume that the three conditions are satisfied.
  Let $\varepsilon>0$, and pick $R$ and $\rho>0$ so that
  $\int_{\R^d}\iverson{\abs{x}>R} \, \norm{f(x)}^p\,dx < \varepsilon^p$
  ($L^p$-equivanishing)
  and
  $\int_{\R^d}\norm{f({x+h})-f(x)}^p\,dx < \varepsilon^p$ whenever $\abs{h} < \rho$
  ($L^p$-equicontinuity).

  Let $Q \subset \R^d$ be a closed cube with diameter less than $\rho$,
  centered at the origin.
  Let $Q_i$, $i=1$, \ldots, $N$, be non-overlapping
  (in the sense of pairwise disjoint interiors)
  translates of $Q$
  whose union contains
  $\Set{x \in \R^d \given \abs{x}<R}$.
  Define the map
  \begin{equation*}
    P \colon L^p(\R^d; B) \to B^N,\quad
    (Pf)_i = f_i := \abs{Q}^{-1} \int_{Q_i} f \, dx.
  \end{equation*}
  Note that $\abs{Q_i}=\abs{Q}$, so the integral is an average over $Q_i$.
  By total boundedness in the mean, $\Set{Pf \given f \in \mathcal{F}}$ is totally bounded.
  Note that
  \begin{align*}
    \norm{f(x)-f_i}^p
    &= \norm[\Big]{\abs{Q}^{-1}\int_{Q_i} \pars[\big]{f(x)-f(y)}\,dy}^p \\
    &\le \pars[\Big]{\abs{Q}^{-1}\int_{Q_i}\norm{f(x)-f(y)}\,dy}^p
    \le \abs{Q}^{-1} \int_{Q_i}\norm{f(x)-f(y)}^p\,dy
  \end{align*}
  using first the integral triangle inequality and then Jensen's inequality, and hence
  \begin{align*}
    \sum_{i=1}^{N} \int_{Q_i} \norm{f(x)-f_i}^p\,dx
    &\le \abs{Q}^{-1} \sum_{i=1}^{N} \int_{Q_i} \int_{Q_i} \norm{f(x)-f(y)}^p\,dy\,dx \\
    &\le \abs{Q}^{-1} \sum_{i=1}^{N} \int_{Q_i} \int_{2Q} \norm{f(x)-f(x+h)}^p \,dh \,dx \\
    &= \abs{Q}^{-1} \int_{2Q} \sum_{i=1}^{N} \int_{Q_i} \norm{f(x)-f(x+h)}^p \,dx \,dh \\
    &\le \abs{Q}^{-1} \int_{2Q} \int_{\R^d} \norm{f(x)-f(x+h)}^p \,dx \,dh < 2^d \varepsilon^p,
  \end{align*}
  at the end using the $L^p$-equicontinuity inequality
  and the fact that $\abs{h} < \rho$ when $h \in 2Q$.

  Now consider two functions $f$, $g \in \mathcal{F}$.
  Using the $L^p$-equivanishing inequality, the triangle inequality in $B$,
  Minkowski's inequality, and the above estimate, we find
  \begin{align*}
    \norm{f-g}_p
    &< 2\varepsilon + \pars[\bigg]{\sum_{i=1}^{N} \int_{Q_i} \norm{f(x)-g(x)}^p\,dx }^{1/p} \\
    &\le 2\varepsilon + \pars[\bigg]{\sum_{i=1}^{N} \int_{Q_i} \pars[\big]{
      \norm{f(x)-f_i}+\norm{f_i-g_i}+\norm{f_i(x)-g(x)}}^p\,dx }^{1/p} \\
    &\le 2\varepsilon
      + \pars[\bigg]{\sum_{i=1}^{N} \int_{Q_i} \norm{f(x)-f_i}^p\,dx }^{1/p}
      + \pars[\bigg]{\sum_{i=1}^{N} \int_{Q_i} \norm{f_i-g_i}^p\,dx }^{1/p} \\
    &\phantom{{}\le 2\varepsilon} + \pars[\bigg]{\sum_{i=1}^{N} \int_{Q_i} \norm{g_i-g(x)}^p\,dx }^{1/p} \\
    &< (2+2^{1+d/p})\varepsilon + \pars[\bigg]{\abs{Q} \sum_{i=1}^{N} \norm{f_i-g_i}^p }^{1/p} .
  \end{align*}
  Thus, if $\pars[\big]{\sum_{i=1}^{N} \norm{f_i-g_i}^p }^{1/p} < \abs{Q}^{-1/p} \varepsilon$
  then $\norm{f-g}_p<(3+2^{1+d/p})\varepsilon$.
  By Lemma \ref{lemma:tb} and the total boundedness of $\Set{Pf \given f \in \mathcal{F}}$,
  $\mathcal{F}$ is totally bounded.
\end{proof}

Directly proving that a family of functions
is totally bounded in the mean can be difficult.
However, a common scenario considers $L^p$-functions
with values in a compactly embedded subspace.
As mentioned in the introduction,
this has since been generalized,
replacing the subspace by a cone.
Once the usefulness of this generalization is realized,
the proof turns out to require little extra work
compared to the original setting.

We start with some definitions.
Note that we deviate from the terminology commonly seen in the literature
(indicated in parentheses below),
which we find cumbersome and somewhat confusing.

\begin{definitions}
  A \emph{cone} (nonnegative cone) is a nonempty subset $X$ of some vector space
  so that $ax \in X$ whenever $a \in [0,\infty)$ and $x \in X$.

  A \emph{gauge} (seminorm) on a cone $X$
  is a map $\gauge{\dott}\colon X \to [0,\infty)$
  so that $\gauge{ax}=a\gauge{x}$ whenever $x \in X$ and $a \ge 0$,
  and $\gauge{x}>0$ if $x \ne 0$.
  A cone with a gauge is called a \emph{gauged cone}
  (seminormed non-negative cone).

  A gauged cone $X$ in a Banach space is called
  a \emph{compactly gauged cone}
  (compactly embedded seminormed non-negative cone (!))
  if $X_r := \Set{x \in X \given \gauge{x} \le r}$ is precompact for $r=1$
  – and hence for all $r>0$, since $X_r=rX_1$.
\end{definitions}

Now assume that $X$ is a compactly gauged cone in a Banach space $B$.
Then there is a constant $C$ such that $\norm{x} \le C\gauge{x}$ for any $x \in X$.
We may assume without loss of generality that $C=1$, so $\norm{\dott} \le \gauge{\dott}$.
Define $L^p(\R^d;X)$ to be the set of all functions
$f \colon \R^d \to X$
such that $f$ is Bochner measurable as a function into $B$,
and $\gauge{f}$ is measurable as well,
with $\int_{\R^d} \gauge{f}^p \,dx < \infty$.
Write $\gauge{f}_p=\pars[\big]{\int_{\R^d}\gauge{f}^p\,dx}^{1/p}$.

The measurability of $\gauge{f}$
does not follow automatically from Bochner measurability,
as is shown by the following simple example:
Let $B = \C$ as a real vector space,
let $X$ be the upper half plane,
let $\psi \colon (0,\pi) \to (0,1)$ be a non-measurable function,
put $\gauge{te^{i\theta}}=t\psi(\theta)$ for $t \ge 0$ and $\theta \in (0,\pi)$,
and define $f(\theta)=e^{i\theta}$ for $\theta \in (0,\pi)$.

We can now state the following result.

\begin{theorem} \label{thm:ccc}
  Let $X$ be a compactly gauged cone in a Banach space $B$,
  and let $\mathcal{F} \subset L^p(\R^d;X)$ be bounded
  \textup(i.e., there is a uniform bound on $\gauge{f}_p$ for all $f \in \mathcal{F}$\textup).
  Assume further that $\mathcal{F}$ is $L^p$-equivanishing
  and $L^p$-equicontinuous in $L^p(\R^d;B)$.
  Then $\mathcal{F}$ is totally bounded in $L^p(\R^d;B)$.
\end{theorem}

We need a definition and some lemmas for the proof.

\begin{definition} \label{def:u-int}
A set $\mathcal{F} \subset L^p(\Omega,\mu;B)$ is called \emph{uniformly $L^p$-integrable}
if for each $\varepsilon>0$ there exists some $r$ so that
\begin{equation*}
  \int_{\Omega} \iverson{\norm{f} > r} \, \norm{f}^p \,d\mu < \varepsilon \qquad(f \in \mathcal{F}).
\end{equation*}
\end{definition}
With this definition we get the following result.
\begin{lemma} \label{lemma:tbui}
  A totally bounded subset of $L^p(\Omega,\mu;B)$ is uniformly $L^p$-integrable.
\end{lemma}
We omit the easy proof,
hinting only that it is true for a singleton set,
hence for a finite set.
Now use total boundedness to approximate the given set by a finite set.
See also \cite{van-neerven-2014}.

\begin{lemma} \label{lemma:ui}
  Assume that $\mathcal{F} \subset L^p(\R^d;B)$ is
  bounded, $L^p$-equivanishing,
  and $L^p$-equicontinuous.
  Then $\mathcal{F}$ is uniformly $L^p$-integrable.
\end{lemma}
\begin{proof}
  First, note that the set $\norm{\mathcal{F}}:=\Set{\norm{f} \given f \in \mathcal{F}}$
  is also bounded in $L^p(\R^d)$
  and $L^p$-equicontinuous –
  the latter follows from
  $\abs[\big]{\norm{f(x+h)}-\norm{f(x)}} \le \norm{f(x+h)-f(x)}$. The set
  $\norm{\mathcal{F}}$ is $L^p$-equivanishing, because $\mathcal{F}$ is.
  And since bounded subsets of $\R$ are totally bounded,
  $\norm{\mathcal{F}}$ is totally bounded by Theorem \ref{thm:kras}.
  (We could also use the Kolmogorov–Riesz theorem \cite[Thm.~5]{h-o-h-2010} here.)
  By Lemma~\ref{lemma:tbui}, $\norm{\mathcal{F}}$ is uniformly $L^p$-integrable,
  and hence (trivially) so is $\mathcal{F}$.
\end{proof}

\begin{lemma} \label{lemma:uix}
  If $\mathcal{F} \subset L^p(\R^d;B)$ is uniformly $L^p$-integrable,
  then for each $\varepsilon>0$ there exists $\delta>0$ so that
  $\int_D \norm{f(x)}^p \,dx < \varepsilon$
  whenever $D \subset \R^d$ is measurable with $\abs{D} < \delta$.
\end{lemma}

\begin{proof}
  Simply write
  \begin{equation*}
    \int_D \norm{f(x)}^p \,dx
    = \int_D \iverson{\norm{f(x)} \le r} \norm{f(x)}^p \,dx
    + \int_D \iverson{\norm{f(x)} > r} \norm{f(x)}^p \,dx,
  \end{equation*}
  pick $r$ so that the last integral is less than $\varepsilon/2$
  for all $f \in \mathcal{F}$ (with $D=\R^d$),
  and then note that the first integral is at most $\delta r^p$.
  Choosing $\delta=\varepsilon/(2r^p)$ yields the desired inequality.
\end{proof}

\begin{proof}[Proof of Theorem \ref{thm:ccc}]
  In light of Theorem \ref{thm:kras},
  we only need to show that $\mathcal{F}$ is totally bounded in the mean.

  First, take any $\varepsilon>0$.
  Combining Lemmas \ref{lemma:ui} and \ref{lemma:uix},
  there is some $\delta>0$ so that $\int_D \norm{f(x)}^p \,dx < \varepsilon^p$
  whenever $\abs{D}<\delta$.

  Write $\gauge{f}_p \le M$ for all $f \in \mathcal{F}$.
  Then $\abs{\Set{x \in \R^d\given \gauge{f(x)}>r}} < (M/r)^p$.
  Hence, if $r>M/\delta^{1/p}$,
  we get $\norm{\iverson{\gauge{f}>r}\, f}_p < \varepsilon$.
  That is, $\norm{f-\tilde f}_p < \varepsilon$,
  where we write $\tilde f = \iverson{\gauge{f} \le r}\,f$.

  Next, note that
  $\tilde f$ takes values in the totally bounded set $X_r$.
  The convex hull $\co X_r$ is totally bounded as well.
  Consider a bounded measurable set $E \subset \R^d$.
  The mean value $\abs{E}^{-1}\int_E \tilde f\,dx$
  belongs to the closure $\cco X_r$ of $\co X_r$,
  which is also totally bounded.

  From the above estimates,
  the integral $\int_E f\,dx$
  has distance less than $\varepsilon$ to the compact set $\abs{E}\cco X_r$.
  Since this is so for any $\varepsilon>0$
  (note that $r$ depends on $\varepsilon>0$),
  the set $\Set{\int_E f\,dx \given f \in \mathcal{F}}$
  is totally bounded.
  This completes the proof.
\end{proof}

The first condition of Theorem \ref{thm:ccc} allows us
to somewhat relax the $L^p$-equicontinuity condition:
We merely need to assume this condition in $L^p(\R^d;Y)$,
where $B$ is continuously embedded in a Banach space $Y$.
The resulting setting can be summarized by the formula
$X \Subset B \subseteq Y$,
where $X$ is a compactly gauged cone in $B$
and $B$ is embedded in $Y$.

However, a yet more general assumption,
first identified in \cite{chen-jungel-liu-2014},
avoids even the need for an embedding,
and instead has $B$ and $Y$ be Banach spaces embedded in some common space,
and $X \subset B \cap Y$ with $X \Subset B$.
This requires one more technical assumption,
detailed in the lemma below.
The assumption clearly holds when $B$ is embedded in $Y$,
and also in the very common situation where both spaces
are embedded in a common topological vector space,
such as a space of vector-valued distributions.

We will write $\norm{\dott}$ for the norm on $B$,
and $\altnorm{\dott}$ for the norm on $Y$.
Correspondingly,
we also write $\norm{\dott}_p$ and $\altnorm{\dott}_p$
for the norms in $L^p(\R^d;B)$ and $L^p(\R^d;Y)$
respectively.

\begin{lemma}[{\cite[Lemma 4]{chen-jungel-liu-2014}}] \label{lemma:cone-ehrling}
  Let $(B,\norm{\dott})$ and $(Y,\altnorm{\dott})$
  be Banach spaces which are subspaces of the same
  vector space,
  and let $X \subset B$ be a compactly gauged cone in $B$.
  Assume that there is no sequence in $B \cap Y$ converging in $Y$ to $0$
  and in $B$ to some non-zero vector.
  Then for each $\eta>0$ there exists $c>0$
  such that
  \begin{equation*}
    \norm{u-v} \le \eta\pars[\big]{\gauge{u}+\gauge{v}}
    + c \altnorm{u-v}\qquad\text{for all $u$, $v \in X$.}
  \end{equation*}
\end{lemma}
\begin{proof}
  First, note the following reformulation of the stated sequence
  condition:

  \textit{Each non-zero $x \in B$ has a neighborhood $U$ in $B$
    so that $\altnorm{\dott}$ has a positive lower bound on $U \cap Y$.}
  (If $U \cap Y = \emptyset$, we may take that lower bound to be $+\infty$.)
  It follows that
  $\altnorm{\dott}$ has a positive lower bound on $K \cap Y$
  for any compact set $K \subset B \setminus \Set{0}$.

  Now let $K$ be the closure in $X$ of the set
  $\Set{u-v \given u, v \in X \text{ and } \gauge{u}+\gauge{v} \le 1}$,
  which is totally bounded in $B$ because $\Set{u \in X \given \gauge{u} \le 1}$
  is totally bounded and subtraction is uniformly continuous.
  Thus $K$ is compact in $B$,
  and so is $L = \Set{w \in K \given \norm{w} \ge \eta}$.
  Since $0 \notin L$,
  $\altnorm{\dott}$ has a positive lower bound on $L \cap Y$,
  so for a sufficiently large $c$,
  $\norm{w} \le \eta + c \altnorm{w}$
  for all $w \in L$, and hence for all $w \in K$.
  That is,
  \begin{equation*}
    \norm{u-v} \le \eta
    + c \altnorm{u-v}\qquad\text{for all $u$, $v \in K$ with $\gauge{u}+\gauge{v} \le 1$.}
  \end{equation*}
  Replacing arbitrary $u$, $v \in X$ by
  $u/\pars[\big]{\gauge{u}+\gauge{v}}$ and
  $v/\pars[\big]{\gauge{u}+\gauge{v}}$ respectively
  yields the desired result.
\end{proof}

\begin{theorem}[compare Thm.\ \ref{thm:cjl} below] \label{thm:cjl+}
  Let $B$ and $Y$ be Banach spaces which are subspaces of the same
  vector space,
  and let $X \subset B$ be a compactly gauged cone in $B$.
  Assume that there is no sequence in $B \cap Y$ converging in $Y$ to $0$
  and in $B$ to some non-zero vector.
  Further, assume that $\mathcal{F} \subset L^p(\R^d;X) \cap L^p(\R^d;Y)$
  is bounded in $L^p(\R^d;X)$,
  $L^p$-equivanishing in $L^p(\R^d;B)$,
  and $L^p$-equicontinuous in $L^p(\R^d;Y)$,
  where $1 \le p < \infty$.
  Then $\mathcal{F}$ is totally bounded in $L^p(\R^d;B)$.
\end{theorem}

\begin{proof}
  We only need to prove that $\mathcal{F}$
  is $L^p$-equicontinuous in $L^p(\R^d;B)$,
  and then Theorem \ref{thm:ccc} takes care of the rest.

  To this end, let $\varepsilon>0$,
  pick $\eta>0$ (to be determined later), and pick a constant $c$
  as in the statement of Lemma \ref{lemma:cone-ehrling}.
  If $f \in \mathcal{F}$ and $h \in \R^d$ then
  \begin{align*}
    \norm{\sigma_hf-f}_p
    &=\pars[\Big]{\int_{\R^d} \norm[\big]{f(x+h)-f(x)}^p\,dx}^{1/p} \\
    &\le\pars[\Big]{\int_{\R^d} \pars[\Big]{
      \eta\pars[\big]{\gauge{f(x+h)}+\gauge{f(x)}}
      +c\altnorm[\big]{f(x+h)-f(x)}
      }^p\,dx}^{1/p} \\
    &\le 2\eta\gauge{f}_p + c\altnorm{\sigma_hf-f}_p.
  \end{align*}
  Since $\mathcal{F}$ is bounded in $L^p(\R^d;X)$,
  we can pick $\eta$ to ensure that $2\eta\gauge{f}_p<\varepsilon$ for all $f \in \mathcal{F}$.
  After this is done, $c$ is now fixed.
  If $\abs{h}$ is small enough, we have $c\altnorm{\sigma_hf-f}_p<\varepsilon$
  for all $f \in \mathcal{F}$, since $\mathcal{F}$ is $L^p$-equicontinuous in $L^p(\R^d;Y)$.
  Thus we have $\norm{\sigma_hf-f}<2\varepsilon$,
  and the proof is complete.
\end{proof}

\section{Compactness for Bochner spaces based on bounded intervals} \label{sec:boundint}

We have expressed our results so far
for functions defined on a Euclidean space $\R^d$,
since this seems natural.
However, in the most commonly occurring applications in PDE theory,
an interval $(0,T)$ is used instead.
We can of course embed $L^p(0,T;B)$ in $L^p(\R;B)$
by extending functions to be zero outside $(0,T)$.
All of our notions and results survive this extension intact,
with one exception, namely, the definition of $L^p$-continuity
(Definition \ref{def:Lp-ec}).
The common definition for functions on $(0,T)$ is as follows.
The seemingly minor, but potentially damaging, difference
is the need to restrict the interval of integration
to the part where $\sigma_hf(t)=f(t+h)$ is defined:

\begin{definition}[\cite{krukowski-2023}] \label{def:LP-eci}
  $\mathcal{F} \subset L^p(0,T;B)$ is called
  \emph{$L^p$-equicontinuous} 
  if $\int_0^{T-h}\norm{f(t+h)-f(t)}^p\,dt \to 0$
  when $0 < h \to 0$, uniformly for $f \in \mathcal{F}$.
\end{definition}

If we extend the integral to all of $\R$,
we get Definition \ref{def:Lp-ec}.
However, in doing so, we have to contend with two extra terms
$\int_0^h \norm{f}^p\,dt$ and $\int_{T-h}^T \norm{f}^p\,dt$.
It turns out that these extra terms vanish in the limit
if we assume also that $\mathcal{F}$ is bounded.
The following proposition allows us to apply all the compactness results
of the previous section to families of functions in $L^p(0,T;B)$.

\begin{proposition} \label{prop:contint}
  Assume that $\mathcal{F} \subset L^p(0,T;B)$ is bounded
  and $L^p$-equicontinuous according to Definition \ref{def:LP-eci}.
  Then, if we extend each function to $\R$
  by setting it zero outside $(0,T)$,
  the resulting set is $L^p$-equicontinuous
  according to Definition \ref{def:Lp-ec}.
\end{proposition}

The proof depends on a slight detour:

\begin{lemma} \label{lemma:normtrans}
  If $\mathcal{F} \subset L^p(0,T;B)$ $(1 \le p < \infty)$ is bounded
  and $L^p$-equicontinuous,
  then $\norm{\mathcal{F}}^p$ is $L^1$-equicontinuous.
\end{lemma}
\begin{proof}
  Assume $1 < p < \infty$, and let $q$ be the conjugate exponent.
  (The case $p=1$ is immediate.)
  Writing $a \vee b$ for the maximum of $a$ and $b$,
  and assuming $0<h<T$, we estimate
  \begin{multline*}
    \int_0^{T-h}\abs[\big]{\norm{f(t+h)}^p-\norm{f(t)}^p}\,dt \\
    \begin{aligned}
      \qquad
    &\le p\int_0^{T-h}\pars[\big]{\norm{f(t+h)}\vee \norm{f(t)}}^{p-1}
    \abs[\big]{\norm{f(t+h)}-\norm{f(t)}}\,dt \\
    &\le p\pars[\Big]{\int_0^{T-h}\pars[\big]{\norm{f(t+h)}\vee \norm{f(t)}}^{(p-1)q}\,dt}^{1/q}\cdot
      \pars[\Big]{\int_0^{T-h}\abs[\big]{\norm{f(t+h)}-\norm{f(t)}}^p\,dt}^{1/p} \\
    &\le p\pars[\Big]{\int_0^{T-h}\pars[\big]{\norm{f(t+h)} + \norm{f(t)}}^p\,dt}^{1/q}\cdot
      \pars[\Big]{\int_0^{T-h}\norm{f(t+h)-f(t)}^p\,dt}^{1/p}
    \end{aligned}
  \end{multline*}
  using the Hölder inequality and noting that $(p-1)q=p$.
  The first integral on the final line is uniformly bounded
  for $f \in \mathcal{F}$,
  and the desired result follows.
\end{proof}

\begin{proof}[Proof of Proposition \ref{prop:contint}]
  We only need to show that
  $\int_0^h \norm{f(t)}^p\,dt \to 0$ and
  $\int_{T-h}^T \norm{f(t)}^p\,dt \to 0$
  as $0<h \to 0$, uniformly for $f \in \mathcal{F}$.
  (Proving the sufficiency of this claim is left to the reader.)
  We only prove the former;
  the latter is shown similarly.

  By Lemma \ref{lemma:normtrans},
  we may as well replace $\mathcal{F}$ by $\norm{\mathcal{F}}^p$,
  and $p$ by $1$;
  i.e., we assume that $B=\R$, $p=1$
  and that $\mathcal{F}$ consists of non-negative functions.

  We shall prove the contrapositive:
  Assume that $\varepsilon>0$ and that for every $h>0$
  there is some $f \in \mathcal{F}$ with $\int_0^h f\,dt \ge \varepsilon$,
  and also that
  $\mathcal{F}$ is $L^p$-equicontinuous according to Definition \ref{def:LP-eci}.

  We shall prove that then $\mathcal{F}$ is unbounded.

  To this end, pick any $\eta>0$
  and let $\rho \in (0,T)$ be so that
  $\int_0^{T-h}\norm{f(t+h)-f(t)}\,dt < \eta$ whenever $0<h<\rho$.
  Next pick any natural number $n$,
  let $h \in (0,\rho)$ with $(n+1)h<T$,
  and let $f \in \mathcal{F}$ with $\int_0^h f\,dt \ge \varepsilon$.
  Write $f_i(t)=f(ih+t)$, for $t\in(0,h)$ and $i=0$, $1$, \ldots, $n$.
  Then
  \begin{align*}
    \int_0^T f\,dt
    &\ge \sum_{i=1}^n \norm{f_i}_1 \\
    &= n\norm{f_0}_1 + \sum_{i=0}^{n-1} (n-i)
      \pars[\big]{\norm{f_{i+1}}_1-\norm{f_{i}}_1} \\
    &\ge n\varepsilon - n \sum_{i=0}^{n-1}\abs[\big]{\norm{f_{i+1}}_1-\norm{f_{i}}_1}
     \ge n\varepsilon - n \sum_{i=0}^{n-1}\norm{f_{i+1}-f_{i}}_1 \\
    &\ge n\varepsilon - n \int_0^{T-h}\abs{f(t+h)-f(t)}\,dt
     \ge n(\varepsilon-\eta).
  \end{align*}
  Thus merely picking $\eta<\varepsilon$ and $n$ large
  (forcing $h$ to be small),
  we get arbitrarily large values for $\norm{f}_1$ with $f \in \mathcal{F}$.
\end{proof}

Combining Prop.~\ref{prop:contint} and Thm.~\ref{thm:cjl+}
immediately yields the following:

\begin{theorem}[{\cite[Thm.~1]{chen-jungel-liu-2014}}] \label{thm:cjl}
  Let $B$ and $Y$ be Banach spaces which are subspaces of the same
  vector space,
  and let $X \subset B$ be a compactly gauged cone in $B$.
  Assume that there is no sequence in $B \cap Y$ converging in $Y$ to $0$
  and in $B$ to some non-zero vector.
  Further, assume that $\mathcal{F} \subset L^p(0,T;X) \cap L^p(0,T;Y)$
  is bounded in $L^p(0,T;X)$,
  $L^p$-equivanishing in $L^p(0,T;B)$,
  and $L^p$-equicontinuous in $L^p(0,T;Y)$,
  where $1 \le p < \infty$.
  Then $\mathcal{F}$ is totally bounded in $L^p(0,T;B)$.
\end{theorem}

\noindent {\bf Acknowledgments.}
The authors thank Endre S\"uli for helpful discussion regarding the paper
by Ehrling \cite{ehrling-1954},
and the anonymous referee
for pointing out a flaw in our original proof of Theorem \ref{thm:ccc}.


\end{document}